\documentclass[a4paper,10pt]{article}
\usepackage{amssymb,amsmath,amsthm,amsfonts}
\usepackage{bookmark}
\usepackage{mathrsfs}
\usepackage{multirow}
\usepackage{graphicx}
\usepackage{authblk}
\usepackage{indentfirst}
\usepackage{multicol}
\usepackage{tabu}
\usepackage{url}
\usepackage{fancyhdr}
\usepackage[numbers]{natbib}
\usepackage[all]{xy}
\usepackage{mathtools}
\usepackage[english]{babel}
\usepackage{colortbl}
\usepackage{caption}
\usepackage{hyperref}\usepackage{subfigure}
\usepackage{tikz}\usepackage{float}\usepackage{mathrsfs}

\newtheorem{theorem}{Theorem}[section]
\newtheorem{proposition}[theorem]{Proposition}
\newtheorem{lemma}[theorem]{Lemma}

\newtheorem{definition}{Definition}[section]
\newtheorem{example}{Example}[section]
\newtheorem{remark}{Remark}[section]

\usepackage{xr}
\makeatletter
    \newcommand*{\addFileDependency}[1]{
    \typeout{(#1)}
    \@addtofilelist{#1}
    \IfFileExists{#1}{}{\typeout{No file #1.}}
    }
\makeatother

\title{The subdivision of hypergraphs}

\author[1,2]{Jian Liu}
\author[3,4]{Ran Liu}
\author[3]{Jie Wu \thanks{Corresponding author: wujie@bimsa.cn}}
\affil[1]{Mathematical Science Research Center, Chongqing University of Technology, Chongqing 400054, China}
\affil[2]{Department of Mathematics, Michigan State University, MI, 48824, USA}
\affil[3]{Yanqi Lake Beijing Institute of Mathematical Sciences and Applications, Beijing 101408, China}
\affil[4]{School of Mathematical Sciences, Beihang University, Beijing, 100191, China}

\makeatletter
    \renewcommand*{\@fnsymbol}[1]{\ensuremath{\ifcase#1\or \dagger\or *\or *\or
   \mathsection\or \else\@ctrerr\fi}}
\makeatother
\date{}

\begin{document}
    \maketitle

    \paragraph{Abstract}

Hypergraphs, as a generalization of simplicial complexes, have long been a subject of interest in their geometric interpretation. The subdivision of simplicial complexes can, to some extent, provide insights into the geometry of simplicial complexes. In this paper, we introduce the concept of the subdivision of hypergraphs. Notably, the subdivision of hypergraphs can be reduced to the subdivision of simplicial complexes. Moreover, we prove that the subdivision of hypergraphs has the topological invariance with respect to the embedded homology.




\section{Introduction}\label{section:introduction}

As is said in \cite{battiston2020networks}, ``Although simplicial complexes overcome some of the problems encountered by other lower dimensional representations, they are still quite limited ... hypergraphs provide the most general and unconstrained description of higher-order interactions.''
Mathematicians have dedicated significant research efforts to exploring the combinatorial properties of hypergraphs. More recently, attention has shifted towards the topological aspects of hypergraphs.
In \cite{chung1992laplacian}, the authors studied the homology and Laplacian of $k$-uniform hypergraphs. In \cite{buzaglo2013topological}, the authors introduce the hypergraph topology, a distinct topology defined on a hypergraph.
Recently, S. Bressan, J. Li, S. Ren, and J. Wu  introduced the concept of embedded homology of hypergraphs \cite{bressan2019embedded}. The idea of embedded homology in hypergraphs can be traced back to the GLMY theory \cite{grigor2019homology,grigor2012homologies,grigor2018path,grigor2017homologies}.  Nevertheless, a geometric understanding of hypergraph remains an ongoing research topic.

An (abstract) simplicial complex always has an underlying space or a geometric realization. Moreover, the simplicial approximation theorem asserts that any arbitrary continuous map between geometric simplicial complexes can be approximated, in a suitable sense, by a simplicial map \cite{munkres2018elements}. The simplicial approximation heavily depends on the subdivision of simplicial complexes, which implies the geometry of simplicial complexes.

Classical result asserts that there is a correspondence $\xymatrix{\mathbf{SimpCpx}\ar@<0.5ex>[r] & \mathbf{Pos}\ar@<0.5ex>[l]}$ between the category of simplicial complexes and the category of posets. A poset can be endowed with an Alexandrov topology \cite{alexandroff1937Diskrete}. Moreover, one can obtain functors
\begin{equation*}
\xymatrix{\mathscr{X}: \mathbf{SimpCpx}\ar@<0.5ex>[r] & \mathbf{T_{0}A}:\mathscr{K}\ar@<0.5ex>[l]}
\end{equation*}
between the category of simplicial complexes and the category of $T_{0}$ Alexandrov spaces ($A$-spaces) \cite{mccord1966singular}. The functor $\mathrm{sd}=\mathscr{K}\mathscr{X}:\mathbf{SimpCpx}\to \mathbf{SimpCpx}$ is the barycentric subdivision functor of simplicial complexes. Any finite simplicial complex can be modelled by the inverse limit of a subdivision system \cite{clader2009inverse,thibault2013homotopy}, i.e., the space $|\mathscr{K}(X)|$ has the homotopy type of the inverse limit of the directed system
\begin{equation*}
  \xymatrix{
  \cdots\ar@{->}[r]^-{\inf}&(\mathrm{sd}^{2}X)^{\mathrm{op}}\ar@{->}[r]^-{\inf}&(\mathrm{sd}X)^{\mathrm{op}}\ar@{->}[r]^-{\inf}&X^{\mathrm{op}},
  }
\end{equation*}
where $X$ is a locally finite $A$-space. Inspired by this idea, the subdivision of hypergraphs has the potential to reveal the geometric encoding of hypergraphs. In the literature \cite{banerjee2023adjacency,iradmusa2020subdivision}, the authors have independently provided combinatorial definitions for the subdivision of hypergraphs. We aim to generalize the subdivision of simplicial complexes and offer a definition for hypergraph subdivision that incorporates a more geometric interpretation.
In \cite{potvin2023hypergraphs}, the author regards a hypergraph as a poset and defines hypergraph subdivision as the order complex of the poset. However, this definition imitates the form of subdivision of simplicial complexes, and the subdivision of a hypergraph results in a simplicial complex. Besides, an intriguing phenomenon emerges: a high-dimensional hyperedge can be subdivided into lower-dimensional simplices, which deviates from our geometric intuition.

In this paper, we will introduce the subdivision of hypergraphs, which possesses three distinct advantages, leading us to believe that it is a satisfactory definition:
\begin{itemize}
  \item The subdivision of hypergraphs coincides with the subdivision of simplicial complexes when the hypergraphs are reduced to simplicial complexes;
  \item The subdivision of hypergraphs aligns with our geometric intuition and can be constructed based on the poset structure;
  \item The subdivision of hypergraphs is a functor on the category of hypergraphs and exhibits topological invariance with respect to the embedded homology of hypergraphs.
\end{itemize}

A graded poset is a poset $X$ equipped with a rank function $\rho:X\to \mathbb{N}$ such that if $y$ covers $x$, then $\rho(y)=\rho(x)+1$. A marked graded poset, denoted as $(X,S)$, is a graded poset $X$ with a subset $S\subseteq X$. We consider the category $\mathbf{GrPos^{+}}$  whose objects are the marked graded posets, and whose morphisms are the compatible morphisms of marked graded posets. Then there is a functor $\mathscr{H}:\mathbf{GrPos^{+}}\to \mathbf{Hyp}$ from the category of marked graded posets to the category of hypergraphs (Proposition \ref{proposition:functor}). On the other hand, there is a functor $\mathscr{P}^{+}:\mathbf{Hyp}\to \mathbf{GrPos^{+}}$ from the category of hypergraphs to the category of marked posets (Proposition \ref{proposition:functor_marked}). We define the subdivision of hypergraphs as a functor
\begin{equation*}
  \mathrm{sd}=\mathscr{H}\mathscr{P}^{+}:\mathbf{Hyp}\to\mathbf{Hyp}
\end{equation*}
on the category of hypergraphs (Definition \ref{definition:subdivision}).
Let $\mathcal{H}$ be a hypergraph, and let $\Delta \mathcal{H}=\{\tau\neq\emptyset|\tau\subseteq \sigma\text{ for some }\sigma\in\mathcal{H}\}$ be the simplicial closure of $\mathcal{H}$. It is worth noting that $\Delta (\mathrm{sd}(\mathcal{H}))=\mathrm{sd}(\Delta\mathcal{H})$, indicating the compatibility between the subdivision of hypergraphs and the subdivision of simplicial complexes (Proposition \ref{proposition:subdivision_closure}). Besides, we present criteria for determining whether an element belongs to the subdivision of a hypergraph and provide methods for computing the subdivision of hypergraphs. Our main result is the topological invariance of the subdivision of hypergraphs, which is stated as follows:
\begin{theorem}
There is a natural isomorphism of embedded homology of hypergraphs
\begin{equation*}
  H_{\ast}(\mathrm{sd}): H_{\ast}(\mathcal{H})\stackrel{\cong}{\rightarrow} H_{\ast}(\mathrm{sd}(\mathcal{H})).
\end{equation*}
\end{theorem}

In the next section, we review fundamental concepts related to the subdivision of simplicial complexes and the embedded homology of hypergraphs. Section \ref{section:marked} provides the construction of hypergraphs from marked posets. In Section \ref{section:subdivision}, we introduce the subdivision of hypergraphs.  At last, we present the proof of our main theorem in Section \ref{section:homology_invariance}.

\section{Preliminaries}\label{section:Preliminary}
\subsection{Subdivision of simplicial complexes}

Let $V$ be a non-empty finite set. An abstract simplicial complex is defined as a non-empty subset $\mathcal{K}$ of the power set $\mathbf{P}(V)$ of $V$ that satisfies the condition: If $\sigma$ is an element of $\mathcal{K}$, then every non-empty subset of $\sigma$ must also belong to $\mathcal{K}$. In contrast, a hypergraph $\mathcal{H}$ on $V$ is a nonempty subset of $\mathbf{P}(V)$. The subdivision of hypergraphs heavily depends on the idea of the subdivision of simplicial complexes.
In this section, we will recall the fundamental concepts related to the subdivision of simplicial complexes.

Let $X$ be a poset. We always denote $\leq$ as the partial order of a poset $X$. As usual, the notion $<$ is the relation on $X$ such that $x<y$ if and only if $x\leq y$ and $x\neq y$. Given a poset $X$, we always have an abstract simplicial complex $\mathscr{F}(X)$ with $n$-simplices given by the sets $\{x_{0},x_{1},\dots,x_{n}\}$ satisfying $x_{0}<x_{1} <\dots< x_{n}$. On the other hand, for an abstract simplicial complex $\mathcal{K}$, we can obtain a poset $\mathscr{P}(\mathcal{K})$ whose elements are the simplices of $\mathcal{K}$, and the partial order is given by $\tau\leq \sigma$ if $\tau\subseteq \sigma$. Thus, we have two functors
\begin{equation*}
\xymatrix{\mathscr{P}: \mathbf{SimpCpx}\ar@<0.5ex>[r] & \mathbf{Pos}:\mathscr{F}\ar@<0.5ex>[l]}
\end{equation*}
between the category of simplicial complexes and the category of posets. The functor $\mathrm{sd}=\mathscr{F}\mathscr{P}:\mathbf{SimpCpx}\to \mathbf{SimpCpx}$ is the subdivision functor on the category of simplicial complexes. It is worth noting that the subdivision functor $\mathscr{F}\mathscr{P}$ is identified with the subdivision functor $\mathscr{K}\mathscr{X}$ mentioned in the introduction. It is a well-known fact that the subdivision has the topological invariance.
\begin{theorem}[\cite{munkres2018elements}]
There is a quasi-isomorphism of chain complexes
\begin{equation*}
  \rho:C_{\ast}(\mathcal{K})\rightarrow C_{\ast}(\mathrm{sd}(\mathcal{K})),
\end{equation*}
that is, $H_{p}(\mathcal{K})\cong H_{p}(\mathrm{sd}(\mathcal{K}))$ for $p\geq 0$.
\end{theorem}
The abstract simplicial complex $\mathcal{K}$ has a geometric realization $|\mathcal{K}|$. There is a classical construction of the barycentric subdivision of geometric simplicial complexes. It is well-known that $\mathrm{sd}(|\mathcal{K}|)=|\mathrm{sd}(\mathcal{K})|$, which says that the subdivision of abstract simplicial complexes coincides with the barycentric subdivision of geometric simplicial complexes. In the following sections, we will introduce the concept of subdivision of hypergraphs, which can be seen as a generalization of the subdivision of abstract simplicial complexes.

\subsection{Embedded homology}

Now, we will review the fundamental concept of embedded homology, which defines the homology of a graded module embedded into a chain complex. Embedded homology provides us with a topological perspective on hypergraphs. From now on, the ground ring $R$ is always assumed to be a commutative ring with unit.

Let $C_{\ast}=(C_{n})_{n\geq 0}$ be a chain complex over $R$. Let $D_{\ast}=\{D_{n}\}_{n\geq 0}$ be a graded $R$-submodule of $C_{\ast}$. We denote the pair $(D_{\ast},C_{\ast})$ as the \emph{embedded complex}. An \emph{embedded map} $f:(D_{\ast},C_{\ast})\rightarrow (D'_{\ast+i},C'_{\ast+i})$  of degree $i$ is a graded $R$-module homomorphism
\begin{equation*}
  f_{n}:C_{n}\rightarrow C'_{n+i}
\end{equation*}
satisfying $f(D_{n})\subseteq D'_{n+i}$.
A \emph{morphism of embedded complexes} $f=\{f_{n}\}_{n\geq 0}:(D_{\ast},C_{\ast})\rightarrow (D'_{\ast},C'_{\ast})$ is an embedded map of degree zero such that $df=fd$ on $C_{\ast}$. Here, $d$ always denotes the differential on the corresponding chain complex.

Let $(D_{\ast},C_{\ast})$ be an embedded complex. The \emph{infimum chain complex} $\mathrm{Inf}_{\ast}(D_{\ast},C_{\ast})$ of $(D_{\ast},C_{\ast})$ is given by $\mathrm{Inf}_{n}(D_{\ast},C_{\ast})=D_{n}\cap d^{-1}D_{n-1}=\{x\in D_{n}|dx\in D_{n-1}\}$. Then we have a functor $\mathrm{Inf}_{\ast}:\mathbf{Emb}\to \mathbf{Chain}$ from the category of embedded chain complexes to the category of chain complexes.
\begin{definition}[\cite{bressan2019embedded}]
The \emph{embedded homology} is a functor $H\circ \mathrm{Inf}_{\ast}:\mathbf{Emb}\to \mathbf{Mod}_{R}$ from the category of embedded chain complexes to the category of $R$-modules.
\end{definition}
By observing the definition, we deduce the following lemma.
\begin{lemma}\label{lemma:isomorphism}
Let $C_{\ast}$ and $C_{\ast}'$ be chain complexes. Assume that $D_{\ast}\subseteq C_{\ast}\cap C_{\ast}'$. Then we have
\begin{equation*}
  \mathrm{Inf}_{n}(D_{\ast},C_{\ast})=\mathrm{Inf}_{n}(D_{\ast},C_{\ast}'),\quad n\geq 0.
\end{equation*}
\end{lemma}

\begin{example}
Let $\mathcal{H}$ be a hypergraph, and let $\Delta\mathcal{H}$ be the simplicial closure of $\mathcal{H}$. Suppose that $D_{\ast}(\mathcal{H})$ is the free $\mathbb{Z}$-module generated by the hyperedges in $\mathcal{H}$. Then $(D_{\ast}(\mathcal{H}),C_{\ast}(\Delta\mathcal{H}))$ is an embedded complex. The embedded homology of $\mathcal{H}$ is given by
\begin{equation*}
  H_{n}(\mathcal{H})=H_{n}(D_{\ast}(\mathcal{H}),C_{\ast}(\Delta\mathcal{H})),\quad n\geq 0.
\end{equation*}
Lemma \ref{lemma:isomorphism} shows that the choice of the simplicial closure $\Delta\mathcal{H}$ is reasonable. Indeed, if we choose another simplicial complex $\mathcal{K}$ containing $\mathcal{H}$, then we have $H_{n}(D_{\ast}(\mathcal{H}),C_{\ast}(\Delta\mathcal{H}))=H_{n}(D_{\ast}(\mathcal{H}),C_{\ast}(\mathcal{K}))$ for $n\geq 0$.
\end{example}

Two morphisms $f,g:(D_{\ast},C_{\ast})\rightarrow (D'_{\ast},C'_{\ast})$ of embedded complexes are \emph{homotopic}, denoted by $f\simeq g$, if there is an embedded map $h:(D_{\ast},C_{\ast})\rightarrow (D'_{\ast+1},C'_{\ast+1})$ of degree $1$ such that $f-g=dh+hd$. A morphism $f:(D_{\ast},C_{\ast})\rightarrow (D'_{\ast},C'_{\ast})$ is a \emph{chain homotopy equivalence} if there exists a morphism $g:(D'_{\ast},C'_{\ast})\rightarrow (D_{\ast},C_{\ast})$ such that $g\circ f\simeq \mathrm{id}$ and $f\circ g\simeq \mathrm{id}$.

\begin{proposition}\label{proposition:homotopy}
Suppose $f,g:(D_{\ast},C_{\ast})\rightarrow (D'_{\ast},C'_{\ast})$ are homotopic. Then the induced morphisms $$f^{\flat},g^{\flat}:\mathrm{Inf}_{\ast}(D_{\ast},C_{\ast})\to \mathrm{Inf}_{\ast}(D_{\ast}',C_{\ast}')$$ are homotopic. Specifically, we have $H_{n}(f)=H_{n}(g):H_{n}(D_{\ast},C_{\ast})\rightarrow H_{n}(D'_{\ast},C'_{\ast})$ for $n\geq 0$.
\end{proposition}
\begin{proof}
Let $f-g=dh+hd$ be an embedded chain homotopy. Since $f$ is an embedded chain morphism, we have $f(D_{n})\subseteq D'_{n}$. For each $x\in D_{n}\cap d^{-1}D_{n-1}$, we obtain
\begin{equation*}
  df(x)=fd(x)\subseteq f(D_{n-1})\subseteq D'_{n-1},
\end{equation*}
which shows that $f(x)\in d^{-1}D'_{n-1}$. Thus we have $f(D_{n}\cap d^{-1}D_{n-1})\subseteq D_{n}'\cap d^{-1}D_{n-1}'$. Hence, $f^{\flat}:\mathrm{Inf}_{\ast}(D_{\ast},C_{\ast})\to \mathrm{Inf}_{\ast}(D'_{\ast},C'_{\ast})$ is a morphism of chain complexes. Similarly, $g^{\flat}:\mathrm{Inf}_{\ast}(D_{\ast},C_{\ast})\to \mathrm{Inf}_{\ast}(D'_{\ast},C'_{\ast})$ is a morphism of chain complexes.

Since $h$ is an embedded map of degree 1, we have $h(D_{n-1})\subseteq D'_{n}$. For each $x\in D_{n}\cap d^{-1}D_{n-1}$, we have $hd(x)\subseteq h(D_{n-1})\subseteq D'_{n}$. It follows that
\begin{equation*}
  dh(x)=f(x)-g(x)-hd(x)\subseteq D'_{n}.
\end{equation*}
Hence, we obtain $h(x)\subseteq d^{-1}D'_{n}$, which implies $h(x)\subseteq D'_{n+1}\cap d^{-1}D'_{n}$. Hence the map $h^{\flat}=h|_{\mathrm{Inf}_{\ast}(D_{\ast},C_{\ast})}:\mathrm{Inf}_{\ast}(D_{\ast},C_{\ast})\to \mathrm{Inf}_{\ast+1}(D'_{\ast},C'_{\ast})$ is an $R$-module homomorphism of degree 1. Thus, we obtain
\begin{equation*}
  f^{\flat}-g^{\flat}=dh^{\flat}+h^{\flat}d
\end{equation*}
on $\mathrm{Inf}_{\ast}(D_{\ast},C_{\ast})$, which is the desired result.
\end{proof}

\section{Marked posets}\label{section:marked}

Just as in the construction of simplicial complexes on posets, we will construct hypergraphs on marked posets. Furthermore, we provide some equivalent conditions for determining whether an element is in the hypergraph. However, the hypergraph construction on the category of marked posets is not always a functor. As an alternative, we will explore the functoriality of hypergraph construction within its subcategory, specifically, the category with compatible morphisms.

\subsection{Hypergraph construction on marked posets}

We will begin by reviewing some fundamental knowledge about posets. Let $X$ be a poset. We say \emph{$y$ covers $x$} if $x<y$ and there is no element $z$ such that $x<z<y$. An element $x\in X$ is \emph{initial} if there is no element covered by $x$.
For a subset $Y$ of $X$, the subset of $X$ \emph{covered by $Y$} is given by
\begin{equation*}
  \mathcal{C}_{X}(Y)=\{x\in X|x \text{ is covered by some element in }Y\}.
\end{equation*}
Let $\mathcal{P}_{X}(Y)=\{x\in X|x<y \text{ for some }y\in Y\}$. Let $\mathcal{I}_{X}(Y)$ be the set of initial elements in $\mathcal{P}(Y)$. The sets $\mathcal{P}_{X}(Y),\mathcal{I}_{X}(Y)$ and $\mathcal{C}_{X}(Y)$ a role in describing the construction of hypergraphs from a marked poset. If there's no ambiguity, we often denote these sets as $\mathcal{P}(Y)$, $\mathcal{I}(Y)$, and $\mathcal{C}(Y)$ for convenience.

\begin{definition}
A \emph{marked poset} $(X,S)$ is a poset $X$ equipped with a subset $S\subseteq X$, which is called the \emph{marked subset}. A \emph{morphism of marked posets} $f:(X,S)\to (Y,T)$ consists of a morphism of poset $f:X\to Y$ such that $f(S)\subseteq T$.
\end{definition}

Recall that $\mathscr{F}(X)$ is an abstract simplicial complex. For the sake of convenience, we always denote the element $\{x_{0},x_{1},\dots,x_{n}\}$ in $\mathscr{F}(X)$ as $x_{0}x_{1}\cdots x_{n}$ for $x_{0}<x_{1}<\cdots<x_{n}$. We can obtain a poset $\mathscr{N}(X)=\mathscr{P}\mathscr{F}(X)$ with the partial order such that $x_{0}x_{1}\cdots x_{n} \leq y_{0}y_{1}\cdots y_{n}$ if and only if $ x_{i}\leq y_{i}$ for $i=0,1,\dots,n$. Let $\mathscr{N}(X,S)=\{x_{0}x_{1}\cdots x_{n}\in \mathscr{N}(X)|x_{n}\in S\}$. Then $\mathscr{N}(X,S)$ is a subset of $\mathscr{N}(X)$. An initial element in $\mathscr{N}(X,S)$ is called \emph{$S$-successive}.

\begin{definition}
Let $(X,S)$ be a marked poset. The \emph{hypergraph construction} $\mathscr{H}(X,S)$ is a subset of $\mathscr{N}(X,S)$ defined by the induction that $\sigma\in \mathscr{H}(X,S)$ if $\sigma$ is $S$-successive or $\mathcal{C}(\{\sigma\})\subseteq \mathscr{H}(X,S)$.
\end{definition}

The following propositions provide us with alternative descriptions of the hyperedges in the construction $\mathscr{H}(X,S)$ on marked posets.

\begin{proposition}\label{proposition:description_p}
Let $\sigma\in \mathscr{N}(X,S)$. Then $\sigma\in \mathscr{H}(X,S)$ if and only if $\mathcal{P}(\{\sigma\})\subseteq \mathscr{H}(X,S)$.
\end{proposition}
\begin{proof}
``$\Rightarrow$''. Suppose $\sigma\in \mathscr{H}(X,S)$. If $\sigma$ is $S$-successive, the assertion is obviously true. Now, we consider the case that $\sigma$ is not an $S$-successive element. For any element $\tau\in \mathcal{P}(\{\sigma\})$, since $X$ is finite, we can find a sequence $\tau<\sigma_{k}<\cdots<\sigma_{1}<\sigma$ in such a way that each element in the sequence strict covers the one before it. Thus, $\sigma\in \mathscr{H}(X,S)$ implies $\sigma_{1}\in \mathscr{H}(X,S)$. By induction, we obtain $\tau\in \mathscr{H}(X,S)$.

``$\Leftarrow$''. Suppose $\mathcal{P}(\{\sigma\})\subseteq \mathscr{H}(X,S)$. If $\sigma$ is $S$-successive, then $\sigma\in \mathscr{H}(X,S)$. Otherwise, we have $\tau\in \mathscr{H}(X,S)$ for any $\tau<\sigma$. It follows that $\mathcal{C}(\{\sigma\})\subseteq \mathscr{H}(X,S)$. Thus, one has $\sigma\in \mathscr{H}(X,S)$.
\end{proof}

\begin{proposition}\label{proposition:initial}
Let $\sigma\in \mathscr{N}(X,S)$. Then $\sigma\in \mathscr{H}(X,S)$ if and only if $\mathcal{I}(\{\sigma\})\subseteq \mathscr{N}(X,S)$.
\end{proposition}
\begin{proof}
``$\Rightarrow$''. Note that $\mathcal{I}(\{\sigma\})\subseteq \mathcal{P}(\{\sigma\})$. By Proposition \ref{proposition:description_p}, one has $\mathcal{I}(\{\sigma\})\subseteq \mathscr{H}(X,S)\subseteq \mathscr{N}(X,S)$.

``$\Leftarrow$''. We will employ proof by contradiction. Suppose there is an initial element $\tau_{0}$ in $\mathcal{I}({\sigma})$ that does not belong to $\mathscr{N}(X,S)$. Then we can find a sequence $\tau_{0}<\sigma_{k}<\cdots<\sigma_{1}<\sigma$ in such a way that each element in the sequence strict covers the one before it. By definition, we have $\tau_{0}\notin \mathscr{H}(X,S)$ because it is not $S$-successive. In view of $\tau_{0}\in C(\{\sigma_{k}\})$, one has $\sigma_{k}\notin \mathscr{H}(X,S)$. By induction, we obtain $\sigma\notin \mathscr{H}(X,S)$. This contradicts to our assumption.
\end{proof}
\subsection{The functorial property of hypergraph construction}

A \emph{graded poset} is a poset $X$ equipped with a rank function $\rho:X\to \mathbb{N}$ such that if $y$ covers $x$, then $\rho(y)=\rho(x)+1$. A \emph{morphism $f:X\to Y$ of graded posets} is a map of posets such that if $y$ covers $x$, then $\rho(f(y))\leq\rho(f(x))+1$. A morphism $f:X\to Y$ of graded posets is called \emph{compatible} if $\mathcal{C}(f(x))\subseteq f(\mathcal{C}(x))$ for any $x\in X$. It is worth noting that the compatible property indicates that $\rho(f(y))\leq\rho(f(x))+1$ for $x\in \mathcal{C}(y)$.

\begin{definition}
A \emph{marked graded poset} $(X,S)$ is a graded poset $X$ equipped with a subset $S\subseteq X$. A \emph{morphism of marked graded posets} $f:(X,S)\to (Y,T)$ consists of a morphism of graded poset $f:X\to Y$ such that $f(S)\subseteq T$. We say $f:(X,S)\to (Y,T)$ is \emph{compatible} if $f:X\to Y$ is compatible.
\end{definition}

\begin{lemma}
Let $f:(X,S)\to (Y,T)$ be a compatible morphism of marked graded posets. Then the map $\mathscr{H}(f):\mathscr{H}(X,S)\to \mathscr{H}(Y,T)$ given by $\mathscr{H}(f)(\sigma)=f(\sigma)$ is a morphism of hypergraphs.
\end{lemma}
\begin{proof}
Let $\sigma=\{x_{0},x_{1},\dots,x_{n}\}\in \mathscr{H}(X,S)$. We have $x_{n}\in S$. It follows that $f(x_{n})\in T$. Thus, one has $f(\sigma)\in \mathscr{N}(Y,T)$.
If $\sigma$ is $S$-successive, by definition, the set $f(\sigma)$ is a $T$-successive element. Then we have $f(\sigma)\in \mathscr{H}(Y,T)$. Suppose $\sigma$ is not $S$-successive. If $f(\sigma)$ is $T$-successive, we have $f(\sigma)\in \mathscr{H}(Y,T)$. Otherwise, for any $\tau_{1}\in \mathcal{C}(\{f(\sigma)\})$, due to the compatibility of $f$, there exists $\sigma_{1}\in \mathcal{C}(\{\sigma\})$ such that $f(\sigma_{1})=\tau_{1}$. By definition, we have $\sigma_{1}\in \mathscr{H}(X,S)$. It follows that $\tau_{1}\in \mathscr{N}(Y,T)$. If $\tau_{1}$ is $T$-successive, we have $\tau_{1}\in \mathscr{H}(Y,T)$. Otherwise, for any $\tau_{2}\in \mathcal{C}(\{\tau_{1}\})$, there exists $\sigma_{2}\in \mathcal{C}(\{\sigma_{1}\})$ such that $f(\sigma_{2})=\tau_{2}$. Due to the finiteness of the posets and by induction, we can find an integer $k$ such that $f(\sigma_{k})=\tau_{k}$ and $\tau_{k}$ is $T$-successive. It follows that $\tau_{k}\in \mathscr{H}(Y,T)$.
Because for any $\tau_k\in C({\tau_{k-1}})$, we have $\tau_k\in \mathscr{H}(Y,T)$, it follows that $\tau_{k-1}\in \mathscr{H}(Y,T)$. By induction, we obtain $\tau\in \mathscr{H}(Y,T)$. The desired result follows.
\end{proof}

Consider the category $\mathbf{GrPos^{+}}$ of marked graded posets whose objects are marked graded posets, and whose morphisms are the compatible morphisms of marked graded posets.

\begin{proposition}\label{proposition:functor}
The construction $\mathscr{H}:\mathbf{GrPos^{+}}\to \mathbf{Hyp}$ is a functor from the category of marked graded posets to the category of hypergraphs.
\end{proposition}
\begin{proof}
The proof can be conducted step by step according to the definition.
\end{proof}

\section{Subdivision of hypergraphs}\label{section:subdivision}

In this section, we introduce a functor $\mathscr{P}^{+}:\mathbf{Hyp}\to\mathbf{GrPos^{+}}$ from the category of hypergraphs to the category of marked graded posets. The subdivision of hypergraphs is defined by the functor $\mathscr{H}\mathscr{P}^{+}$.
We provide examples to illustrate the computation of the subdivision of hypergraphs and explore the method for carrying out the computation.

\subsection{The subdivision functor}
Let $\mathcal{H}$ be a hypergraph, and let $\Delta\mathcal{H}$ be the simplicial closure of $\mathcal{H}$. We can obtain a graded marked poset $(\mathscr{P}(\Delta\mathcal{H}),\mathcal{H})$ with the rank function $\rho:\mathscr{P}(\Delta\mathcal{H})\to \mathbb{N}$ given by $\rho(\sigma)=\#|\sigma|-1$. Here, $\#|\sigma|$ denotes the number of elements in $\sigma$. Moreover, for a map of hypergraphs $\phi:\mathcal{H}\to \mathcal{H}'$, there is a morphism of graded marked posets $\mathscr{P}^{+}(\phi):(\mathscr{P}(\Delta\mathcal{H}),\mathcal{H})\to (\mathscr{P}(\Delta\mathcal{H}'),\mathcal{H}')$ given by $\mathscr{P}(\phi)(\sigma)=\phi(\sigma)$. The map $\mathscr{P}(\phi):\mathscr{P}(\Delta\mathcal{H})\to \mathscr{P}(\Delta\mathcal{H}')$ is a simplicial map.

\begin{lemma}
The morphism $\mathscr{P}^{+}(\phi):(\mathscr{P}(\Delta\mathcal{H}),\mathcal{H})\to (\mathscr{P}(\Delta\mathcal{H}'),\mathcal{H}')$ induced by $\phi:\mathcal{H}\to \mathcal{H}'$ is compatible.
\end{lemma}
\begin{proof}
For any element $\sigma\in \Delta\mathcal{H}$ and any $\tau\in \mathcal{C}(\{\phi(\sigma)\})$, we have $\tau=\partial_{i}\phi(\sigma)$ for some face map $\partial_{i}$ on $\Delta\mathcal{H}$. It follows that $\tau=\phi(\partial_{i}\sigma)\subseteq \phi(\mathcal{C}(\{\sigma\}))$. Thus the morphism $\mathscr{P}(\phi)$ is compatible.
\end{proof}
It is not difficult to verify the functorial property of the construction $\mathscr{P}^{+}$.
\begin{proposition}\label{proposition:functor_marked}
The construction $\mathscr{P}^{+}:\mathbf{Hyp}\to \mathbf{GrPos^{+}}$ is a functor from the category of hypergraphs to the category of marked posets.
\end{proposition}

\begin{proposition}\label{proposition:equal_hypergraph}
Let $\mathcal{H}$ be a hypergraph. Suppose $\mathcal{K}$ is a simplicial complex such that $\mathcal{H}\subseteq \mathcal{K}$. Then we have
\begin{equation*}
  \mathscr{H}(\mathscr{P}(\Delta\mathcal{H}),\mathcal{H})=\mathscr{H}(\mathscr{P}(\mathcal{K}),\mathcal{H})
\end{equation*}
\end{proposition}
\begin{proof}
Note that $\Delta\mathcal{H}$ is a sub complex of $\mathcal{K}$. We have a morphism $j:(\mathscr{P}(\Delta\mathcal{H}),\mathcal{H})\hookrightarrow (\mathscr{P}(\mathcal{K}),\mathcal{H})$ of graded marked posets. It induces an inclusion of hypergraphs $\mathscr{H}(j):\mathscr{H}(\mathscr{P}(\Delta\mathcal{H}),\mathcal{H})\hookrightarrow \mathscr{H}(\mathscr{P}(\mathcal{K}),\mathcal{H})$. Moreover, the morphism $\mathscr{H}(j)$ is an identity map, as each hyperedge $\sigma_{0}\sigma_{1}\cdots\sigma_{n}$ in $\mathscr{H}(\mathscr{P}(\mathcal{K}),\mathcal{H})$ has the preimage $\sigma_{0}\sigma_{1}\cdots\sigma_{n}$ in $\mathscr{H}(\mathscr{P}(\Delta\mathcal{H}),\mathcal{H})$.
\end{proof}

Proposition \ref{proposition:equal_hypergraph} demonstrates that the construction of hypergraphs for $(\mathscr{P}(\mathcal{K}), \mathcal{H})$ does not depend on the choice of the simplicial complex $\mathcal{K}$. This guarantees that the choice of $\Delta \mathcal{H}$ does not alter the construction of hypergraphs. Consequently, we can define the subdivision of hypergraphs as follows.

\begin{definition}\label{definition:subdivision}
The \emph{subdivision of hypergraphs} is the functor $\mathrm{sd}=\mathscr{H}\mathscr{P}^{+}:\mathbf{Hyp}\to \mathbf{Hyp}$ is a functor on the category of hypergraphs.
\end{definition}

The following proposition shows that the subdivision of hypergraphs can be reduced to the subdivision of simplicial complexes.
\begin{proposition}\label{proposition:subdivision_closure}
$\Delta (\mathrm{sd}(\mathcal{H}))=\mathrm{sd}(\Delta\mathcal{H})$.
\end{proposition}
\begin{proof}
If $\sigma_{0}\sigma_{1}\cdots\sigma_{n}\in \mathrm{sd}(\Delta\mathcal{H})$, then $\sigma_{n}\in \Delta\mathcal{H}$. By definition, $\sigma_{n}$ is a subset of some element $\tau_{m}\in \mathcal{H}$.
Then we can extend $\sigma_{0}\sigma_{1}\cdots\sigma_{n}$ to an $\mathcal{H}$-successive element $\tau_{0}\tau_{1}\cdots\tau_{m}\in \mathrm{sd}(\Delta\mathcal{H})$ by adding elements to $\sigma_{0},\sigma_{1},\dots,\sigma_{n}$ such that $\sigma_{0},\sigma_{1},\dots,\sigma_{n}$ is a subsequence of $\tau_{0},\tau_{1},\dots,\tau_{m}$. Since $\tau_{0}\tau_{1}\cdots\tau_{m}$ is $\mathcal{H}$-successive, we have $\tau_{0}\tau_{1}\cdots\tau_{m}\in \mathrm{sd}(\mathcal{H})$. It follows that $\sigma_{0}\sigma_{1}\cdots\sigma_{n}\in\Delta (\mathrm{sd}(\mathcal{H}))$. Thus one has $\mathrm{sd}(\Delta\mathcal{H})\subseteq \Delta (\mathrm{sd}(\mathcal{H}))$.

If $\sigma_{0}\sigma_{1}\cdots\sigma_{n}\in \Delta (\mathrm{sd}(\mathcal{H}))$, then there is an element $\tau_{0}\tau_{1}\cdots\tau_{m}\in\mathrm{sd}(\mathcal{H})$ such that $\sigma_{0},\sigma_{1},\dots,\sigma_{n}$ is a subsequence of $\tau_{0},\tau_{1},\dots,\tau_{m}$. Hence, we have $\tau_{m}\in \mathcal{H}$. It follows that $\sigma_{0},\sigma_{1},\dots,\sigma_{n}\subseteq \tau_{m}$, which implies $\sigma_{0},\sigma_{1},\dots,\sigma_{n}\in\Delta \mathcal{H}$. So we have $\sigma_{0}\sigma_{1}\cdots\sigma_{n}\in \mathrm{sd}(\Delta\mathcal{H})$, implying that $\Delta (\mathrm{sd}(\mathcal{H}))\subseteq \mathrm{sd}(\Delta\mathcal{H})$.
\end{proof}

\subsection{Examples and calculation}

\begin{example}\label{example:subdivision}
Consider the hypergraph $\mathcal{H}=\{\{0\},\{1\},\{0,1\},\{1,2\},\{0,1,2\}\}$. Let $\mathbf{x}=\{0\}\{0,1,2\}$,  $\mathbf{y}=\{1\}\{0,1,2\}$ and $\mathbf{z}=\{2\}\{0,1,2\}$.
\begin{figure}[H]
\centering
\begin{tikzpicture}
\filldraw[color=black!100, fill=black!15,  thick]  (-1,-1)--(-2,0.7)--(-3,-1)   ;
\filldraw[dashed,color=black!100, fill=blue!15,  thick] (-1,-1) -- (-3,-1);
\node [font=\fontsize{8}{6}] (node011) at (-3.3,-1.2){\{0\}};
\node [font=\fontsize{8}{6}] (node012) at (-0.7,-1.2){\{2\}};
\node [font=\fontsize{8}{6}] (node112) at (-1.6,0.7){\{1\}};
\node [font=\fontsize{8}{6}] (node003) at (-3,-0.2){\{0,1\}};
\node [font=\fontsize{8}{6}] (node004) at (-1,-0.2){\{1,2\}};
\node [font=\fontsize{8}{6}] (node005) at (-2,-0.6){\{0,1,2\}};
\draw[fill=black!100](-3,-1) circle(2pt);
\draw[fill=black!0](-1, -1) circle(2pt);
\draw[fill=black!100](-2,0.7) circle(2pt);
\end{tikzpicture}
\qquad\qquad
\begin{tikzpicture}
\filldraw[color=black!100, fill=black!15,  thick]  (-1,-1)--(-2,0.7)--(-3,-1)   ;
\filldraw[dashed,color=black!100, fill=blue!15,  thick] (-1,-1) -- (-3,-1);
\filldraw[color=black!100, fill=blue!15,  thick] (-2,0.7) -- (-2,-0.433);
\filldraw[dashed,color=black!100, fill=blue!15,  thick] (-2,-0.433)-- (-2,-1);
\filldraw[dashed,color=black!100, fill=blue!15,  thick] (-1,-1) -- (-2.5,-0.15);
\filldraw[dashed,color=black!100, fill=blue!15,  thick] (-3,-1) -- (-1.5,-0.15);
\node [font=\fontsize{8}{6}] (node011) at (-3.3,-1.2){\{0\}};
\node [font=\fontsize{8}{6}] (node012) at (-0.7,-1.2){\{2\}};
\node [font=\fontsize{8}{6}] (node112) at (-1.6,0.7){\{1\}};
\draw[fill=black!100](-3,-1) circle(2pt);
\draw[fill=black!0](-1, -1) circle(2pt);
\draw[fill=black!100](-2,0.7) circle(2pt);
\draw[fill=black!0](-2,-1) circle(2pt);
\draw[fill=black!100](-2.5,-0.15) circle(2pt);
\draw[fill=black!0](-1.5,-0.15) circle(2pt);
\draw[fill=black!0](-2,-0.433) circle(2pt);
\end{tikzpicture}
\caption{The left is the original hypergraph, and the right is its subdivision. The dashed lines and the empty points represent missing hyperedges.}
\end{figure}
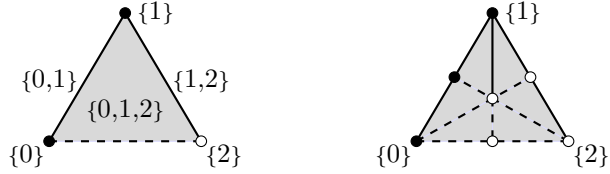
\noindent A straightforward calculation shows that the initial elements in $\mathscr{N}(\mathscr{P}(\Delta\mathcal{H}),\mathcal{H})$ are
\begin{equation*}
  \begin{split}
    &\{0\}\{0,2\}, \{0\}\{0,1\} \leq \mathbf{x}, \\
    &\{1\}\{1,2\},\{1\}\{0,1\}\leq \mathbf{y},  \\
    &\{2\}\{1,2\},\{2\}\{0,2\}\leq \mathbf{z}.
  \end{split}
\end{equation*}
The elements $\{0\}\{0,1\},\{1\}\{1,2\},\{1\}\{0,1\},\{2\}\{1,2\}$ are $\mathcal{H}$-successive, while $\{0\}\{0,2\},\{2\}\{0,2\}$ are not $\mathcal{H}$-successive. By Proposition \ref{proposition:initial}, we have $\mathbf{y}\in \mathrm{sd}(\mathcal{H})$ and $\mathbf{x},\mathbf{z}\notin \mathrm{sd}(\mathcal{H})$. Besides, $\{2\}$ is not an $\mathcal{H}$-successive element, which shows that $\{0,1,2\}\notin\mathrm{sd}(\mathcal{H})$. A similar calculation shows that $\{0,1\}\in \mathcal{H}$.
Together with all the $\mathcal{H}$-successive elements, we obtain
\begin{equation*}
  \begin{split}
    \mathrm{sd}_{0}(\mathcal{H})=&\{\{0\},\{1\},\{0,1\}\}, \\
    \mathrm{sd}_{1}(\mathcal{H})=&\{\{0\}\{0,1\},\{1\}\{0,1\},\{1\}\{1,2\},\{2\}\{1,2\},\{1\}\{0,1,2\}\},\\
    \mathrm{sd}_{2}(\mathcal{H})=&\{\{0\}\{0,1\}\{0,1,2\},\{1\}\{0,1\}\{0,1,2\},\{0\}\{0,2\}\{0,1,2\},\\
    &\{2\}\{0,2\}\{0,1,2\},\{1\}\{1,2\}\{0,1,2\},\{2\}\{1,2\}\{0,1,2\}\}.
  \end{split}
\end{equation*}
\end{example}
\begin{remark}
If $\mathcal{K}$ is a simplicial complex, recall that $\mathrm{sd}(\mathcal{K})=\mathscr{F}\mathscr{P}(\mathcal{K})$. Regard $\mathcal{H}$ as a poset with the partial order $\sigma\leq\tau$ if $\sigma$ is a subset of $\tau$. Then we have an abstract simplicial complex $\mathscr{F}\mathscr{P}(\mathcal{H})$. It is worth noting that there is no inclusion relationship between subdivision $\mathrm{sd}(\mathcal{H})$ and hypergraph $\mathscr{F}\mathscr{P}(\mathcal{H})$. Indeed, consider Example \ref{example:subdivision}.
Note that $\{0,1\}\{0,1,2\}\in \mathscr{F}\mathscr{P}(\mathcal{H})$, but $\{0,1\}\{0,1,2\}\notin \mathrm{sd}(\mathcal{H})$. On the other hand, we observe that
$\{2\}\{1,2\}\in \mathrm{sd}(\mathcal{H})$ while $\{2\}\{1,2\}\notin \mathscr{F}\mathscr{P}(\mathcal{H})$.
\end{remark}

Let $\mathcal{H}$ be a hypergraph. For a hyperedge $\mathbf{x}=\sigma_{0}\sigma_{1}\cdots \sigma_{n}\in \mathscr{N}(\mathscr{P}(\Delta\mathcal{H}))$, let $\mathcal{F}_{k}(\{\mathbf{x}\})$ be the subset of $\mathcal{P}(\{\mathbf{x}\})$ given by the elements of the form $\sigma_{0}\cdots\sigma_{k-1}\tau_{k}\sigma_{k+1}\cdots \sigma_{n}$ for $\tau_{k}\in \mathcal{C}(\{\sigma_{k}\})$. In other words, the element $\sigma_{0}\cdots\sigma_{k-1}\tau_{k}\sigma_{k+1}\cdots \sigma_{n}$ in $\mathcal{F}_{k}(\{\mathbf{x}\})$ is always covered by $\mathbf{x}$ such that $\tau_{k}$ is covered by $\sigma_{k}$. For a subset $S$ of $\mathscr{N}(\mathscr{P}(\Delta\mathcal{H}))$, we denote $\mathcal{F}_{k}(S)=\bigcup\limits_{\mathbf{x}\in S}\mathcal{F}_{k}(\{\mathbf{x}\})$.
Proposition \ref{proposition:initial} provides us with a method for computing the subdivision in terms of $\mathcal{I}({\mathbf{x}})$. From now on, we will delve into the process of computing the set $\mathcal{I}({\mathbf{x}})$.

\begin{example}
Let $\mathcal{H}$ be a hypergraph with vertices $V=\{0,1,2,3,4\}$, and let
\begin{equation*}
  \mathbf{x} =\{0\}\{0,1,2\}\{0,1,2,3,4\}\in \mathscr{N}(\mathscr{F}(\Delta\mathcal{H})).
\end{equation*}
By a straightforward calculation, we have
\begin{equation*}
  \begin{split}
    \mathcal{F}_{4}(\{\mathbf{x}\})=&\{\{0\}\{0,1,2\}\{0,1,2,3\},\{0\}\{0,1,2\}\{0,1,2,4\}\}, \\
     \mathcal{F}_{2}(\mathcal{F}_{4}(\{\mathbf{x}\}))=&\{\{0\}\{0,1\}\{0,1,2,3\},\{0\}\{0,2\}\{0,1,2,3\},\{0\}\{0,1\}\{0,1,2,4\},\\
     &\{0\}\{0,2\}\{0,1,2,4\}\},\\
     \mathcal{F}_{3}(\mathcal{F}_{2}(\mathcal{F}_{4}(\{\mathbf{x}\})))= & \{\{0\}\{0,1\}\{0,1,2\},\{0\}\{0,1\}\{0,1,3\},\{0\}\{0,1\}\{0,1,4\}, \\
    & \{0\}\{0,2\}\{0,1,2\},\{0\}\{0,2\}\{0,2,3\},\{0\}\{0,2\}\{0,2,4\}\},\\
    \mathcal{F}_{3}(\mathcal{F}_{4}(\mathcal{F}_{2}(\{\mathbf{x}\})))= &\{\{0\}\{0,1\}\{0,1,2\},\{0\}\{0,1\}\{0,1,3\},\{0\}\{0,1\}\{0,1,4\}, \\
    & \{0\}\{0,2\}\{0,1,2\},\{0\}\{0,2\}\{0,2,3\},\{0\}\{0,2\}\{0,2,4\}\}.
  \end{split}
\end{equation*}
It follows that $\mathcal{F}_{3}(\mathcal{F}_{2}(\mathcal{F}_{4}(\{\mathbf{x}\})))=\mathcal{F}_{3}(\mathcal{F}_{4}(\mathcal{F}_{2}(\{\mathbf{x}\})))=\mathcal{I}(\{\mathbf{x}\})$.

\end{example}
Let $W_{n}=\{(a_{0},a_{1},\dots, a_{n})\in \mathbb{Z}^{n+1}|0\leq a_{0}< a_{1}< \cdots< a_{n}\}$. Given an element $\omega=(a_{0},a_{1},\dots, a_{n})\in W_{n}$, if $a_{k}>a_{k-1}+1$ for $k>0$ or $a_{k}>0$ for $k=0$, we have
\begin{equation*}
  T_{k}\omega=(a_{0},\dots, a_{k}-1,\dots, a_{n})\in W_{n}.
\end{equation*}
For a given $\omega=(a_{0},a_{1},\dots, a_{n})\in W_{n}$, we consider the sequence
\begin{equation*}
  \omega=\omega_{0}\rightarrow \omega_{1}\rightarrow\dots \rightarrow\omega_{p}=(0,1,\dots, n),\quad p=\sum\limits_{i=0}^{n}(a_{i}-i)
\end{equation*}
with $\omega_{0},\dots,\omega_{p}\in W_{n}$ and $T_{k_{t}}:\omega_{t-1}\rightarrow \omega_{t}$ for some $k_{t}$. The sequence is determined by a family of integers $k_{1},k_{2},\dots,k_{p}$, and we denote the set of all such $(k_{1},k_{2},\dots,k_{p})$ by $U(\omega)$.

\begin{proposition}\label{proposition:determine}
Let $\mathbf{x}=\sigma_{0}\sigma_{1}\cdots\sigma_{q}$ be a hyperedge in $\mathscr{N}(\mathscr{P}(\Delta\mathcal{H}))$, and let $p=\sum\limits_{i=0}^{q}(\dim\sigma_{i}-i)$. For $\omega=(\dim\sigma_{0},\dots,\dim\sigma_{q})$ and each $(k_{1},k_{2},\dots,k_{p})\in U(\omega)$, we have $\mathcal{F}_{k_{p}}\cdots\mathcal{F}_{k_{1}}(\{\mathbf{x}\})=\mathcal{I}(\{\mathbf{x}\})$.
\end{proposition}

\begin{proof}
It is obvious that $\mathcal{F}_{k_{p}}\cdots\mathcal{F}_{k_{1}}(\{\mathbf{x}\})\subseteq\mathcal{I}(\{\mathbf{x}\})$. Now, for any initial element $\mathbf{y}=\xi_{0}\xi_{1}\cdots\xi_{q}\in \mathcal{I}(\{\mathbf{x}\})$, we will prove $\mathbf{y}\in \mathcal{F}_{k_{p}}\cdots\mathcal{F}_{k_{1}}(\{\mathbf{x}\})$ by induction. Suppose that $\mathbf{z}_{s-1}=\tau_{0}\tau_{1}\cdots\tau_{q}\in \mathcal{F}_{k_{s-1}}\cdots\mathcal{F}_{k_{1}}(\{\mathbf{x}\})$ with $\xi_{i}\subseteq \tau_{i}$ for $i=0,\cdots,s-1$. Since $(k_{1},k_{2},\dots,k_{p})\in U(\omega)$, we have that $m=\dim\tau_{k_{s}}>\dim\tau_{k_{s}-1}+1$. Let
\begin{equation*}
  \tau_{k_{s}}=\{v_{0},\cdots,v_{m}\}.
\end{equation*}
We will prove $\tau_{k_{s}-1}\cup \xi_{k_{s}}\subsetneqq\tau_{k_{s}}$. Obviously, we have $\tau_{k_{s}-1}\cup \xi_{k_{s}}\subseteq\tau_{k_{s}}$.
Suppose $\tau_{k_{s}-1}\cup \xi_{k_{s}}=\tau_{k_{s}}$. Since $\mathbf{y}=\xi_{0}\xi_{1}\cdots\xi_{q}$ is an initial element, we have
\begin{equation*}
  \xi_{k_{s}}=\xi_{k_{s}-1}\cup \{v_{\alpha}\}
\end{equation*}
for some $v_{\alpha}\in \tau_{k_{s}}$. It follows that
\begin{equation*}
  \tau_{k_{s}-1}\cup\{v_{\alpha}\}=\tau_{k_{s}-1}\cup\xi_{k_{s}-1}\cup \{v_{\alpha}\}=\tau_{k_{s}-1}\cup\xi_{k_{s}}=\tau_{k_{s}}.
\end{equation*}
Thus we obtain $\tau_{k_{s}-1}=\partial_{\alpha}\tau_{k_{s}}$, which contradicts the condition $\dim\tau_{k_{s}-1}<\dim\tau_{k_{s}}-1$. Hence, we have $\tau_{k_{s}-1}\cup \xi_{k_{s}}\neq \tau_{k_{s}}$. Then there exists an element $v_{t}\in\tau_{k_{s}}\backslash(\tau_{k_{s}-1}\cup \xi_{k_{s}})$. It follows that
\begin{equation*}
  \tau_{0}\subsetneqq \cdots \subsetneqq\tau_{k_{s}-1} \subsetneqq\tau_{k_{s}}\backslash\{v_{t}\} \subsetneqq\cdots\subsetneqq \tau_{q}.
\end{equation*}
Moreover, we can choose
\begin{equation*}
  \mathbf{z}_{s}=\{\tau_{0},\dots,\tau_{k_{s}-1},\tau_{k_{s}}\backslash\{v_{t}\},\tau_{k_{s}+1}\dots,\tau_{q}\}  \in \mathcal{F}_{k_{s}}\cdots\mathcal{F}_{k_{1}}(\{\mathbf{x}\}).
\end{equation*}
Moreover, we have $\mathbf{z}_{s}\in\mathcal{F}_{k_{s}}(\{\mathbf{z}_{s-1}\})$ and $\mathbf{y}\leq \mathbf{z}_{s}$ in $\mathscr{N}(\mathscr{P}(\Delta\mathcal{H}))$.
By  induction, we have $\mathbf{z}_{p}\in \mathcal{F}_{k_{p}}\cdots\mathcal{F}_{k_{1}}(\{\mathbf{x}\})$. Since $\mathbf{z}_{p}$ is an initial element in $\mathscr{N}(\mathscr{P}(\Delta\mathcal{H}))$, we obtain $\mathbf{y}=\mathbf{z}_{p}$. The desired result follows.
\end{proof}
Proposition \ref{proposition:determine} provides a method for computing $\mathcal{I}({\mathbf{x}})$, which is useful in determining whether an element $\mathbf{x}$ belongs to $\mathrm{sd}(\mathcal{H})$. For instance, consider $\mathbf{x}=\sigma_{0}\sigma_{1}\cdots\sigma_{q}$ with $d_{i}=\dim\sigma_{i}$ for $i=0,1,\dots,q$. We choose $(k_{1},k_{2},\dots,k_{p})$ as
\begin{equation*}
(d_{0},d_{0}-1,\dots,1,d_{1},d_{1}-1,\dots,2,\dots,d_{q},d_{q}-1,\dots,q+1),
\end{equation*}
enabling a stepwise computation to obtain $\mathcal{I}({\mathbf{x}})$.
\section{Homological invariance of subdivision}\label{section:homology_invariance}
In this section, we show that he subdivision of hypergraphs has the topological invariance with respect to the embedded homology. We will begin by presenting the main theorem, followed by several lemmas, and finally, provide the proof of the main theorem.

\begin{theorem}\label{theorem:main}
There is a natural isomorphism of embedded homology of hypergraphs
\begin{equation*}
  H_{\ast}(\mathrm{sd}): H_{\ast}(\mathcal{H})\stackrel{\cong}{\rightarrow} H_{\ast}(\mathrm{sd}(\mathcal{H})).
\end{equation*}
\end{theorem}
Let $\mathcal{H}$ be a hypergraphs. For any $\sigma\in \mathcal{H}$, the initial elements in $\mathscr{N}(\mathscr{P}(\Delta\mathcal{H}),\{\sigma\})$ are of the form
\begin{equation*}
  \Delta_{k_{1}k_{2}\cdots k_{n}}^\sigma=(\partial_{k_{n}}\cdots \partial_{k_{1}}\sigma)(\partial_{k_{n-1}}\cdots \partial_{k_{1}}\sigma)\cdots (\partial_{k_{1}}\sigma)\sigma,
\end{equation*}
where $0\leq k_{1}\leq n$, $0\leq k_{2}\leq n-1$, $\cdots$, $0\leq k_{n}\leq 1$. Here, $\partial_{i}$ denotes the face map on the simplicial complex $\Delta\mathcal{H}$. Then $\Delta_{k_{1}k_{2}\cdots k_{n}}^\sigma$ is an $\mathcal{H}$-successive element. Thus we have $\Delta_{k_{1}k_{2}\cdots k_{n}}^\sigma\in \mathrm{sd}(\mathcal{H})$.

\begin{lemma}\label{lemma:rho_morphism}
The $R$-module homomorphism $\rho:(D_{\ast}(\mathcal{H}),C_{\ast}(\Delta\mathcal{H}))\rightarrow (D_{\ast}(\mathrm{sd}(\mathcal{H})),C_{\ast}(\mathrm{sd}(\Delta\mathcal{H})))$ given by
\begin{equation*}
  \rho(\sigma)=\sum_{\mbox{\tiny
    $
\begin{array}{c}
0\leq k_{t}\leq n+1-t,\\
    1\leq t\leq n
\end{array}
$
}}(-1)^{k_{1}+\dots+k_{n}-\frac{n(n+1)}{2}}\Delta_{k_{1}k_{2}\cdots k_{n}}^\sigma.
\end{equation*}
is a morphism of embedded chain complexes.
\end{lemma}
\begin{proof}
Now, we will show $d\rho(\sigma)=\rho(d\sigma)$. By a straightforward calculation, we have
\begin{equation*}
\begin{split}
  d\rho (\sigma)= & \sum_{i=0}^{n}(-1)^{i}\sum_{k_{1},\dots,k_{n}}(-1)^{(\sum\limits_{t=1}^{n}k_{t})-\frac{n(n+1)}{2}}\partial_{i}\Delta_{k_{1}k_{2}\cdots k_{n}}^\sigma \\
    =& \sum_{i=0}^{n-1}\sum_{k_{1},\dots,k_{n}}(-1)^{i+(\sum\limits_{t=1}^{n}k_{t})-\frac{n(n+1)}{2}}\partial_{i}\Delta_{k_{1}k_{2}\cdots k_{n}}^\sigma+\sum_{k_{1},\dots,k_{n}}(-1)^{n+(\sum\limits_{t=1}^{n}k_{t})-\frac{n(n+1)}{2}}\partial_{n}\Delta_{k_{1}k_{2}\cdots k_{n}}^\sigma.
\end{split}
\end{equation*}
Here, the sum $\sum\limits_{k_{1},\dots,k_{n}}$ runs over all the integers $0\leq k_{t}\leq n+1-t$ for $1\leq t\leq n$. Note that
\begin{equation*}
\begin{split}
   & \sum_{k_{1},\dots,k_{n}}(-1)^{i+(\sum\limits_{t=1}^{n}k_{t})-\frac{n(n+1)}{2}}\partial_{i}\Delta_{k_{1}k_{2}\cdots k_{n}}^\sigma \\
   = & \sum_{k_{t},t\neq i,i+1}(-1)^{i+(\sum\limits_{t\neq i,i+1} k_{t})-\frac{n(n+1)}{2}}\left(\sum_{k_{i+1}<k_{i}}(-1)^{k_{i}+k_{i+1}}\partial_{i}\Delta_{k_{1}k_{2}\cdots k_{n}}^\sigma
   +\sum_{k_{i+1}\geq k_{i}}(-1)^{k_{i}+k_{i+1}}\partial_{i}\Delta_{k_{1}k_{2}\cdots k_{n}}^\sigma\right).
\end{split}
\end{equation*}
Since $\partial_{k_{i+1}}\partial_{k_{i}}\sigma=\partial_{k_{i}-1}\partial_{k_{i+1}}\sigma$ for $k_{i+1}<k_{i}$, we have
\begin{equation*}
  \partial_{i}\Delta_{k_{1}k_{2}\cdots k_{n}}^\sigma=\partial_{i}\Delta_{k_{1}\cdots k_{i-1}k_{i+1}(k_{i}-1)k_{i+2}\cdots k_{n}}^\sigma,
\end{equation*}
which implies that
\begin{equation*}
  \sum_{k_{i+1}<k_{i}}(-1)^{k_{i}+k_{i+1}}\partial_{i}\Delta_{k_{1}k_{2}\cdots k_{n}}^\sigma+\sum_{k_{i+1}\geq k_{i}}(-1)^{k_{i}+k_{i+1}}\partial_{i}\Delta_{k_{1}k_{2}\cdots k_{n}}^\sigma=0.
\end{equation*}
Thus we have
\begin{equation*}
  \sum_{i=0}^{n-1}\sum_{k_{1},\dots,k_{n}}(-1)^{i+(\sum\limits_{t=1}^{n}k_{t})-\frac{n(n+1)}{2}}\partial_{i}\Delta_{k_{1}k_{2}\cdots k_{n}}^\sigma=0.
\end{equation*}
A straightforward calculation shows that
\begin{equation*}
  \begin{split}
     d\rho(\sigma)= &\sum_{k_{1},\dots,k_{n}}(-1)^{n+(\sum\limits_{t=1}^{n}k_{t})-\frac{n(n+1)}{2}}\partial_{n}\Delta_{k_{1}k_{2}\cdots k_{n}}^\sigma \\
      =& \sum_{k_{1},\dots,k_{n}}(-1)^{(\sum\limits_{t=1}^{n}k_{t})-\frac{(n-1)n}{2}}(\partial_{k_{n}}\cdots \partial_{k_{1}}\sigma)\cdots (\partial_{k_{1}}\sigma)\\
      =&  \sum_{k_{1}}(-1)^{k_{1}}\sum_{k_{2},\dots,k_{n}}(-1)^{(\sum\limits_{t=2}^{n}k_{t})-\frac{(n-1)n}{2}}\Delta_{k_{2}\cdots k_{n}}^{\partial_{k_{1}}\sigma}\\
      =& \sum_{k_{1}}(-1)^{k_{1}}\rho(\partial_{k_{1}}\sigma)\\
      =&\rho(d\sigma).
  \end{split}
\end{equation*}
Hence, the map $\rho$ is indeed a morphism of embedded complexes.
\end{proof}

Now, define the map $\pi:C_{\ast}(\mathrm{sd}(\Delta\mathcal{H}))\rightarrow C_{\ast}(\Delta\mathcal{H})$ as follows. For an ordered set $\sigma=\{v_{0},v_{1},\dots,v_{s}\}$ in the abstract simplicial complex $\Delta\mathcal{H}$, let $\pi(\sigma)=v_{s}$. For a hyperedge $\sigma_{0}\sigma_{1}\cdots\sigma_{n}$ in $\mathrm{sd}(\Delta\mathcal{H})$, let
\begin{equation*}
  \pi(\sigma_{0}\sigma_{1}\cdots\sigma_{n})=\left\{
        \begin{array}{ll}
          \{\pi(\sigma_{0}),\dots,\pi(\sigma_{n})\}, & \hbox{if $\pi(\sigma_{0}),\dots,\pi(\sigma_{n})$ are different;} \\
          0, & \hbox{otherwise.}
        \end{array}
      \right.
\end{equation*}
\begin{lemma}
We have a morphism $\pi: (D_{\ast}(\mathrm{sd}(\mathcal{H})),C_{\ast}(\mathrm{sd}(\Delta\mathcal{H})))\rightarrow (D_{\ast}(\mathcal{H}),C_{\ast}(\Delta\mathcal{H}))$ of embedded complexes.
\end{lemma}
\begin{proof}
If $\pi(\sigma_{0}),\dots,\pi(\sigma_{n})$ are different, we have
\begin{equation*}
  \begin{split}
\pi d(\sigma_{0}\sigma_{1}\cdots\sigma_{n})=& \pi(\sum_{i=0}^{n}(-1)^{i}\sigma_{0}\cdots\sigma_{i-1}\sigma_{i+1}\cdots\sigma_{n})\\
      =&\sum_{i=0}^{n}(-1)^{i}\{\pi(\sigma_{0}),\cdots,\pi(\sigma_{i}),\pi(\sigma_{i+1})\cdots,\pi(\sigma_{n})\}\\
      =&d\pi (\sigma_{0}\sigma_{1}\cdots\sigma_{n}).
  \end{split}
\end{equation*}
If $\sigma_{t}=\sigma_{t+1}$ for some $t$, a direct computation shows that
\begin{equation*}
  \pi d(\sigma_{0}\sigma_{1}\cdots\sigma_{n})=0=d\pi (\sigma_{0}\sigma_{1}\cdots\sigma_{n}).
\end{equation*}
On the other hand, if $\sigma_{0}\sigma_{1}\cdots\sigma_{n}\in \mathrm{sd}(\mathcal{H})$ and $\pi(\sigma_{0}),\dots,\pi(\sigma_{n})$ are different, we assume that
\begin{equation*}
  \pi (\sigma_{0}\sigma_{1}\cdots\sigma_{n})=\{\pi(\sigma_{0}),\dots,\pi(\sigma_{n})\}=\{v_{d_{0}},\dots,v_{d_{n}}\},
\end{equation*}
where  $d_{i}=\dim\sigma_{i}$ and $d_{0}<d_{1}< \cdots < d_{n}$.
Note that $\{v_{d_{0}}\}=\partial_{0}^{d_{0}}\sigma_{0}$ and each $\{v_{d_{0}},\dots,v_{d_{i}}\}$ can be obtained by face maps of $\sigma_{i}$ in several steps.
Thus we have
\begin{equation*}
  \{v_{d_{0}}\}\{v_{d_{0}},v_{d_{1}}\}\cdots\{v_{d_{0}},\dots,v_{d_{n}}\}\in \mathcal{P}(\{\sigma_{0}\sigma_{1}\cdots\sigma_{n}\}).
\end{equation*}
By Proposition \ref{proposition:description_p}, one has $\{v_{d_{0}}\}\{v_{d_{0}},v_{d_{1}}\}\cdots\{v_{d_{0}},\dots,v_{d_{n}}\}\in \mathrm{sd}(\mathcal{H})$. It follows that $\{v_{d_{0}},\dots,v_{d_{n}}\}\in \mathcal{H}$. So we have $\pi(D_{\ast}(\mathrm{sd}(\mathcal{H})))\subseteq D_{\ast}(\mathcal{H})$.
Thus $\pi$ induces a morphism
\begin{equation*}
  \pi: (D_{\ast}(\mathrm{sd}(\mathcal{H})),C_{\ast}(\mathrm{sd}(\Delta\mathcal{H})))\rightarrow (D_{\ast}(\mathcal{H}),C_{\ast}(\Delta\mathcal{H}))
\end{equation*}
of embedded complexes.
\end{proof}

\begin{proof}[Proof of Theorem \ref{theorem:main}]
We will prove the morphism
\begin{equation*}
  \rho:(D_{\ast}(\mathcal{H}),C_{\ast}(\Delta\mathcal{H}))\rightarrow (D_{\ast}(\mathrm{sd}(\mathcal{H})),C_{\ast}(\mathrm{sd}(\Delta\mathcal{H})))
\end{equation*}
is a chain homotopy equivalence between embedded chain complexes.

Now, a straightforward calculation shows
\begin{equation*}
  \pi(\Delta_{k_{1}k_{2}\cdots k_{n}}^\sigma)=\left\{
                                                \begin{array}{ll}
                                                  \sigma, & \hbox{$k_{1}=1,k_{2}=2,\dots,k_{n}=n$;} \\
                                                  0, & \hbox{otherwise.}
                                                \end{array}
                                              \right.
\end{equation*}
Thus we have
\begin{equation*}
\begin{split}
   \pi\rho(\sigma)=& \pi\left(\sum_{\mbox{\tiny
    $
\begin{array}{c}
0\leq k_{t}\leq n+1-t,\\
    1\leq t\leq n
\end{array}
$
}}(-1)^{k_{1}+\dots+k_{n}-\frac{n(n+1)}{2}}\Delta_{k_{1}k_{2}\cdots k_{n}}^\sigma\right) \\
    =& \sum_{\mbox{\tiny
    $
\begin{array}{c}
0\leq k_{t}\leq n+1-t,\\
    1\leq t\leq n
\end{array}
$
}}(-1)^{k_{1}+\dots+k_{n}-\frac{n(n+1)}{2}}\pi(\Delta_{k_{1}k_{2}\cdots k_{n}}^\sigma)\\
   =&\sigma.
\end{split}
\end{equation*}
It follows that $\pi\rho=\mathrm{id}$.

On the other hand, we define $h:C_{\ast}(\mathrm{sd}(\Delta\mathcal{H}))\rightarrow C_{\ast+1}(\mathrm{sd}(\Delta\mathcal{H}))$ as follows. For a hyperedge $\mathbf{x}=\sigma_{0}\sigma_{1}\cdots\sigma_{n}$ in $\mathrm{sd}(\Delta\mathcal{H})$, let $\tau_{i}=\pi(\sigma_{0}\sigma_{1}\cdots\sigma_{i})$. Then we define
\begin{equation*}
  h(\mathbf{x})=\sum\limits_{i=0}^{n}\sum\limits_{k_{1},\dots,k_{i}}(-1)^{i+\sum\limits_{t=0}^{i}k_{t}-\frac{i(i+1)}{2}}\Delta^{\tau_{i}}_{k_{1},\dots,k_{i}}\sigma_{i}\cdots \sigma_{n},
\end{equation*}
where the sum $\sum\limits_{k_{1},\dots,k_{i}}$ runs over all the integers $0\leq k_{t}\leq i+1-t$ for $1\leq t\leq i$ and $\Delta^{\tau_{i}}_{k_{1},\dots,k_{i}}\sigma_{i}\cdots \sigma_{n}$ denotes the $(n+1)$-hyperedge
\begin{equation*}
  (\partial_{k_{i}}\cdots \partial_{k_{1}}\tau_{i})(\partial_{k_{i-1}}\cdots \partial_{k_{1}}\tau_{i} )\cdots (\partial_{k_{1}}\tau_{i})\tau_{i}\sigma_{i}\cdots \sigma_{n}.
\end{equation*}
Specially, if $\tau_{i}=0$, the corresponding $(n+1)$-hyperedge is reduced to zero. We will prove that $h(D_{\ast}(\mathrm{sd}(\mathcal{H})))\subseteq D_{\ast}(\mathrm{sd}(\mathcal{H}))$. Assume that $\mathbf{x}=\sigma_{0}\sigma_{1}\cdots\sigma_{n}$ is a hyperedge in $\mathrm{sd}(\mathcal{H})$, it suffices to prove $\Delta^{\tau_{i}}_{k_{1}\cdots k_{i}}\sigma_{i}\cdots \sigma_{n}\in \mathrm{sd}(\mathcal{H})$ for $0\leq i\leq n$ and $0\leq k_{t}\leq i+1-t$ for $1\leq t\leq i$. We consider the case that $\tau_{i}$ is nontrivial.
\begin{enumerate}
  \item[(i)] For any initial element
  \begin{equation*}
  \tau_{0}\cdots\tau_{i-1}\tau_{i}\sigma'_{i}\cdots \sigma'_{n}\leq\tau_{0}\cdots\tau_{i-1}\tau_{i}\sigma_{i}\cdots \sigma_{n},
  \end{equation*}
  in $\mathscr{N}(\mathscr{P}(\Delta \mathcal{H}))$, we have
  \begin{equation*}
  \tau_{0}\cdots\tau_{i-1}\tau_{i}\sigma'_{i+1}\cdots \sigma'_{n}\leq \tau_{0}\cdots\tau_{i-1}\tau_{i}\sigma_{i+1}\cdots \sigma_{n}.
  \end{equation*}
  Note that $\tau_{0}\cdots\tau_{i-1}\tau_{i}\sigma_{i+1}\cdots \sigma_{n}\leq\sigma_{0} \cdots  \sigma_{n}$.
  Thus we get
  \begin{equation*}
  \tau_{0}\cdots\tau_{i-1}\tau_{i}\sigma'_{i+1}\cdots \sigma'_{n}\leq \sigma_{0} \cdots  \sigma_{n}.
  \end{equation*}
  Since $\sigma_{0}\cdots \sigma_{n}\in \mathrm{sd}(\mathcal{H})$, we have $\tau_{0}\cdots\tau_{i-1}\tau_{i}\sigma'_{i+1}\cdots \sigma'_{n}\in \mathrm{sd}(\mathcal{H})$.
  By definition, one has $\sigma_{n}'\in \mathcal{H}$.
  It follows that the hyperedge $\tau_{0}\dots\tau_{i}\sigma'_{i}\cdots \sigma'_{n}$ is $\mathcal{H}$-successive. By Proposition \ref{proposition:initial}, we have $\tau_{0}\cdots\tau_{i}\sigma_{i}\cdots \sigma_{n}\in \mathrm{sd}(\mathcal{H})$.
  \item[(ii)] To show $\Delta^{\tau_{i}}_{k_{1}\cdots k_{i}}\sigma_{i}\cdots \sigma_{n}\in \mathrm{sd}(\mathcal{H})$, it suffices to prove $\Delta^{\tau_{i}}_{k_{1}\cdots k_{i}}\sigma_{i}'\cdots \sigma'_{n}\in \mathcal{P}(\mathcal{H})$  for each initial hyperedge $\Delta^{\tau_{i}}_{k_{1}\cdots k_{i}}\sigma_{i}'\cdots \sigma'_{n}\leq\Delta^{\tau_{i}}_{k_{1}\cdots k_{i}}\sigma_{i}\cdots \sigma_{n}$ in $\mathscr{N}(\mathscr{P}(\Delta \mathcal{H}))$. In view of $\tau_{0}\cdots\tau_{i}\sigma_{i}\cdots \sigma_{n}\in \mathrm{sd}(\mathcal{H})$, we have that $\tau_{0}\cdots\tau_{i}\sigma_{i}'\cdots \sigma_{n}'\in \mathrm{sd}(\mathcal{H})$. It follows that  $\sigma_{n}'\in \mathcal{H}$, and then $\Delta^{\tau_{i}}_{k_{1}\cdots k_{i}}\sigma_{i}'\cdots \sigma'_{n}$ is $\mathcal{H}$-successive.
\end{enumerate}
From the above discussion, we obtain an embedded map
\begin{equation*}
  h:(D_{\ast}(\mathrm{sd}(\mathcal{H})),C_{\ast}(\mathrm{sd}(\Delta\mathcal{H})))\rightarrow (D_{\ast+1}(\mathrm{sd}(\mathcal{H})),C_{\ast+1}(\mathrm{sd}(\Delta\mathcal{H}))).
\end{equation*}

Let $\mathbf{x}=\sigma_{0}\sigma_{1}\cdots\sigma_{n}$ be a hyperedge in $\mathrm{sd}(\mathcal{H})$. We will prove $dh+hd=\mathrm{id}-\rho\pi$.\\
Case (I). $\pi(\mathbf{x})$ is nontrivial. By a direct calculation, we have
\begin{equation*}
\begin{split}
  d h(\mathbf{x})= & \sum_{j=0}^{n+1}  \sum\limits_{i=0}^{n}\sum\limits_{k_{1},\dots,k_{i}}(-1)^{i+\sum\limits_{t=0}^{i}k_{t}-\frac{i(i+1)}{2}+j}\partial_{j}(\Delta^{\tau_{i}}_{k_{1},\dots,k_{i}}\sigma_{i}\cdots \sigma_{n}) \\
    =&   \sum\limits_{i=0}^{n}\sum\limits_{k_{1},\dots,k_{i}}\left(\sum_{j=0}^{i}(-1)^{i+\sum\limits_{t=0}^{i}k_{t}-\frac{i(i+1)}{2}+j}(\partial_{j}\Delta^{\tau_{i}}_{k_{1},\dots,k_{i}})\sigma_{i}\cdots \sigma_{n}\right.\\
    &\left.+\sum_{j=i+1}^{n+1}(-1)^{i+\sum\limits_{t=0}^{i}k_{t}-\frac{i(i+1)}{2}+j}\Delta^{\tau_{i}}_{k_{1},\dots,k_{i}}\sigma_{i}\cdots\hat{\sigma}_{j-1}\cdots \sigma_{n}\right)\\
 =&   \sum\limits_{i=0}^{n}\sum\limits_{k_{1},\dots,k_{i}}\sum_{j=0}^{i}(-1)^{i+\sum\limits_{t=0}^{i}k_{t}-\frac{i(i+1)}{2}+j}(\partial_{j}\Delta^{\tau_{i}}_{k_{1},\dots,k_{i}})\sigma_{i}\cdots \sigma_{n}\\
    &+\sum\limits_{i=0}^{n}\sum\limits_{k_{1},\dots,k_{i}}\sum_{j=i}^{n}(-1)^{i+\sum\limits_{t=0}^{i}k_{t}-\frac{i(i+1)}{2}+j+1}\Delta^{\tau_{i}}_{k_{1},\dots,k_{i}}\sigma_{i}\cdots\hat{\sigma}_{j}\cdots \sigma_{n}\\
     =&I_{1}+I_{2}.
\end{split}
\end{equation*}

\begin{equation*}
  \begin{split}
     h(d \mathbf{x}) =& h(\sum_{j=0}^{n}(-1)^{j}\sigma_{0}\cdots\hat{\sigma}_{j}\cdots\sigma_{n}) \\
      =& \sum_{j=0}^{n}\sum\limits_{k_{1},\dots,k_{i}}(-1)^{j}\left(\sum\limits_{i=0}^{j-1}(-1)^{i+\sum\limits_{t=0}^{i}k_{t}-\frac{i(i+1)}{2}}\Delta_{k_{1}\cdots k_{i}}^{\tau_{i}}\sigma_{i}\cdots\hat{\sigma}_{j}\cdots\sigma_{n}          \right.\\
      &+\left. \sum\limits_{i=j}^{n-1}(-1)^{i+\sum\limits_{t=0}^{i}k_{t}-\frac{i(i+1)}{2}}\Delta_{k_{1}\cdots k_{i}}^{\tau_{i+1}^{j}}\sigma_{i+1}\cdots\sigma_{n}            \right)\\
      =&  \sum_{i=0}^{n-1}\sum\limits_{k_{1},\dots,k_{i}}\sum\limits_{j=0}^{i}(-1)^{i+\sum\limits_{t=0}^{i}k_{t}-\frac{i(i+1)}{2}+j}\Delta_{k_{1}\cdots k_{i}}^{\tau_{i+1}^{j}}\sigma_{i+1}\cdots\sigma_{n}       \\
      &+\sum_{i=0}^{n-1}\sum\limits_{k_{1},\dots,k_{i}} \sum\limits_{j=i+1}^{n}(-1)^{i+\sum\limits_{t=0}^{i}k_{t}-\frac{i(i+1)}{2}+j}\Delta_{k_{1}\cdots k_{i}}^{\tau_{i}}\sigma_{i}\cdots\hat{\sigma}_{j}\cdots\sigma_{n}  \\
    =&J_{1}+J_{2}.
  \end{split}
\end{equation*}
Here, $\tau_{i}^{j}=\{\pi(\sigma_{0}),\dots,\pi(\sigma_{j-1}),\pi(\sigma_{j+1}),\dots,\pi(\sigma_{i})\}$. Note that
\begin{equation*}
  I_{2}+J_{2}=-\sum\limits_{i=0}^{n}\sum\limits_{k_{1},\dots,k_{i}}(-1)^{\sum\limits_{t=0}^{i}k_{t}-\frac{i(i+1)}{2}}\Delta^{\tau_{i}}_{k_{1},\dots,k_{i}}\sigma_{i+1}\cdots \sigma_{n},
\end{equation*}
\begin{equation*}
  I_{1}= \sum\limits_{i=0}^{n}(-1)^{i}\rho(d\tau_{i})\sigma_{i}\cdots \sigma_{n},
\end{equation*}
and
\begin{equation*}
\begin{split}
  J_{1}= &  \sum_{i=0}^{n-1}\sum\limits_{k_{1},\dots,k_{i}}\sum\limits_{j=0}^{i}(-1)^{i+\sum\limits_{t=0}^{i}k_{t}-\frac{i(i+1)}{2}+j}\Delta_{k_{1}\cdots k_{i}}^{\tau_{i+1}^{j}}\sigma_{i+1}\cdots\sigma_{n} \\
    =& \sum_{i=0}^{n-1}(-1)^{i}\rho(d\tau_{i+1})\sigma_{i+1}\cdots\sigma_{n}- \sum_{i=0}^{n-1}\sum\limits_{k_{1},\dots,k_{i}}(-1)^{\sum\limits_{t=0}^{i}k_{t}-\frac{i(i+1)}{2}+1}\Delta_{k_{1}\cdots k_{i}}^{\tau_{i+1}^{i+1}}\sigma_{i+1}\cdots\sigma_{n}\\
    =&\sum_{i=1}^{n}(-1)^{i-1}\rho(d\tau_{i})\sigma_{i}\cdots\sigma_{n}+\sum_{i=0}^{n-1}\sum\limits_{k_{1},\dots,k_{i}}(-1)^{\sum\limits_{t=0}^{i}k_{t}-\frac{i(i+1)}{2}}\Delta_{k_{1}\cdots k_{i}}^{\tau_{i}}\sigma_{i+1}\cdots\sigma_{n}.
\end{split}
\end{equation*}
Combined with the above calculations, we get
\begin{equation*}
I_{1}+J_{1}+I_{2}+J_{2}=\sigma_{0}\sigma_{1}\cdots\sigma_{n}-\sum\limits_{k_{1},\dots,k_{n}}(-1)^{\sum\limits_{t=0}^{n}k_{t}-\frac{n(n+1)}{2}}\Delta^{\tau_{n}}_{k_{1},\dots,k_{n}}.
\end{equation*}
On the other hand, we have
\begin{equation*}
    \mathbf{x}-\rho\pi(\mathbf{x})=\sigma_{0}\sigma_{1}\cdots\sigma_{n}-\sum\limits_{k_{1},\dots,k_{n}}(-1)^{\sum\limits_{t=0}^{n}k_{t}-\frac{n(n+1)}{2}}\Delta^{\tau_{n}}_{k_{1},\dots,k_{n}}.
\end{equation*}
It follows that $dh+hd=\mathrm{id}-\rho\pi$.\\
Case (II). $\pi(\mathbf{x})$ is trivial. We assume that $\pi(\sigma_{0}),\pi(\sigma_{1}),\dots,\pi(\sigma_{m})$ are different and $\pi(\sigma_{m})=\pi(\sigma_{m+1})$ for $m<n$. Then $\tau_{i}=0$ for $i\geq m+1$. The calculation is similar to Case (I). A different point is shown as the following.
\begin{equation*}
\begin{split}
  J_{1}= &  \sum_{i=0}^{n-1}\sum\limits_{k_{1},\dots,k_{i}}\sum\limits_{j=0}^{i}(-1)^{i+\sum\limits_{t=0}^{i}k_{t}-\frac{i(i+1)}{2}+j}\Delta_{k_{1}\cdots k_{i}}^{\tau_{i+1}^{j}}\sigma_{i+1}\cdots\sigma_{n} \\
    =& \sum_{i=0}^{m-1}(-1)^{i}\rho(d\tau_{i+1})\sigma_{i+1}\cdots\sigma_{n}- \sum_{i=0}^{m-1}\sum\limits_{k_{1},\dots,k_{i}}(-1)^{\sum\limits_{t=0}^{i}k_{t}-\frac{i(i+1)}{2}+1}\Delta_{k_{1}\cdots k_{i}}^{\tau_{i+1}^{i+1}}\sigma_{i+1}\cdots\sigma_{n}\\
    &+(-1)^{\sum\limits_{t=0}^{m}k_{m}-\frac{m(m+1)}{2}}\Delta_{k_{1}\cdots k_{m}}^{\tau_{m+1}^{m}}\sigma_{m+1}\cdots\sigma_{n}-(-1)^{\sum\limits_{t=0}^{m}k_{m}-\frac{m(m+1)}{2}}\Delta_{k_{1}\cdots k_{i}}^{\tau_{m+1}^{m+1}}\sigma_{m+1}\cdots\sigma_{n}\\
    =&\sum_{i=1}^{m-1}(-1)^{i-1}\rho(d\tau_{i})\sigma_{i}\cdots\sigma_{n}+\sum_{i=0}^{m-1}\sum\limits_{k_{1},\dots,k_{i}}(-1)^{\sum\limits_{t=0}^{i}k_{t}-\frac{i(i+1)}{2}}\Delta_{k_{1}\cdots k_{i}}^{\tau_{i}}\sigma_{i+1}\cdots\sigma_{n}.
\end{split}
\end{equation*}
Here, we use the fact $\tau_{m+1}^{m}=\tau_{m+1}^{m+1}$. A direct computation shows that
\begin{equation*}
  \mathrm{id}-\rho\pi=dh+hd.
\end{equation*}
Finally, by Proposition \ref{proposition:homotopy}, we obtain $H(\pi)H(\rho)=\mathrm{id}$ and $H(\rho)H(\pi)=\mathrm{id}$. Hence $\rho$ is a quasi-isomorphism.  The naturality is from the construction of the subdivision.
\end{proof}

\section{Acknowledgments}
The work was supported by Natural Science Foundation of China (NSFC grant no. 11971144), High-level Scientic Research Foundation of Hebei Province, the start-up research fund from BIMSA.

\bibliographystyle{plain}  
\bibliography{Reference}

\end{document}